\documentclass[11pt,letterpaper]{amsart}

\newtheorem{innercustomgeneric}{\customgenericname}
\providecommand{\customgenericname}{}
\newcommand{\newcustomtheorem}[2]{%
  \newenvironment{#1}[1]
  {%
   \renewcommand\customgenericname{#2}%
   \renewcommand\theinnercustomgeneric{##1}%
   \innercustomgeneric
  }
  {\endinnercustomgeneric}
}

\newcustomtheorem{customthm}{Theorem}

\usepackage[all,cmtip]{xy}

\usepackage{amsthm}
\usepackage{amssymb}
\usepackage{amsmath}
\usepackage{graphicx}
\usepackage{amscd}
\usepackage{amsfonts}
\usepackage{amsbsy}
\usepackage[T1]{fontenc}
\usepackage[english]{babel}

\makeatletter
\@namedef{subjclassname@2020}{\textup{2020} Mathematics Subject Classification}
\makeatother

\usepackage{epsfig}
\usepackage{amssymb}

\newtheorem{theorem}{Theorem}
\newtheorem{cor}{Corollary}
\newtheorem{lem}{Lemma}
\newtheorem{question}{Question}
\newtheorem{definition}{Definition}
\newtheorem{example}{Example}
\newtheorem{remark}{Remark}

\title[Characterizing expansivity through $C^*$-algebras]{Characterizing expansivity through $C^*$-algebras}

\author{S. Bautista}

\address{Department of Mathematics,
National University of Colombia, Bogot\'a-Colombia.}
\email{sbautistad@unal.edu.co}

\author{W. Jung}

\address{Chungnam National University, Daejeon, Republic of Korea}
\email{wcjungdynamics@gmail.com}

\author{C.A. Morales}

\address{Hangzhou International Innovation Institute of Beihang University,  
Hangzhou 311115, China.}
\email{morales@impa.br}


\begin{document}

\begin{abstract}
We study expansive homeomorphisms of a compact metric space \(X\) through the lens of the commutative \(C^*\)-algebra \(C(X)\) of continuous complex-valued functions, viewed as observables of the system. We introduce the notion of expansive observables: elements of \(C(X)\) whose level sets distinguish distinct orbits. We prove that the expansive observables form an F\(_\sigma\)-subalgebra of \(C(X)\), and we characterize them completely for connected equicontinuous homeomorphisms, showing that only constant observables are expansive in this setting. Furthermore, we establish that topologically conjugate homeomorphisms share the same algebra of expansive observables. Using this framework, we show that the set of periodic points intersects at most countably many level sets of any expansive observable. This provides $C^*$-algebraic proofs of well-known facts like for instance that the set of periodic points of an expansive homeomorphism is countable or that the sole continuum exhibiting homeomorphisms which are both expansive and equicontinuous are the degenerated ones. Finally, we prove that no homeomorphism of the circle or the unit interval admits a dense set of expansive observables, yielding a $C^*$-algebraic demonstration of the nonexistence of expansive homeomorphisms in these spaces.
\end{abstract}

\maketitle

\section{Introduction}

\noindent
Much have been written about how a dynamical system represented by a homeomorphism $f:X\to X$ of a metric space $X$ and the C$^*$-algebra $C(X)$ of continuous complex-valued functions on $X$ are related.
This is usually done through the associated 
C$^*$-dynamical system
$(C(X),\alpha)$ where
$$
\alpha:\mathbb{Z}\to Aut(C(X))
$$
is the group homomorphism
$$
\alpha_i(\varphi)=\varphi\circ f^{-i},
\quad
\quad\forall (i,\varphi)\in\mathbb{Z}\times C(X)
$$
and $Aut(C(X))$ is the set of automorphisms of $C(X)$.
These devices were used in the $C^*$-algebra proof of the well-known fact that a homeomorphism of a locally compact metric space has a Borel section if and only if the sets of recurrent and periodic points coincide \cite{st}.
Approach of this kind gave rise the theory of C$^*$-dynamical systems, see the monographs \cite{t1} and \cite{t2}.

In this paper, we will explore a different view about this relationship.
Recall that a homeomorphism $f:X\to X$ is {\em expansive} if
there is a positive number
$\epsilon$ such that if $x,y\in X$ and $d(f^i(x),f^i(y))\leq \epsilon$ for all $i\in\mathbb{Z}$, then $x=y$.
The basic question is:
\begin{question}
Can we use the C$^*$-algebra $C(X)$ to identify expansive homeomorphisms of $X$?
\end{question}

A natural approach is to consider the induced $C^*$ dynamical system $(C(X),\alpha)$ as above yielding the following definition
(related to
{\em strict observability} \cite{AA18} and {\em fiberwise expansivity} \cite{r02}).
For simplicity the elements of $C(X)$ will be referred to as {\em observables}.

\begin{definition}
An observable $\varphi$ is {\em strongly expansive} if there is $\epsilon>0$ (called {\em strong expansivity constant}) such that
$$
x,y\in X\mbox{ and }
|\alpha_i(\varphi)(x)-\alpha_i(\varphi)(y)|\leq\epsilon\quad(\forall i\in\mathbb{Z})\quad\Longrightarrow\quad \varphi(x)=\varphi(y).
$$
\end{definition}

Is there any relationship between the strongly expansive observables and the expansivity of $f$? 
What if every observable is strongly expansive? 
What properties may have $SExp(f)$, the set of strongly expansive observables of a given homeomorphism $f$?
We don't know if a homeomorphism of a compact metric space is expansive if (and also only if) every observable is strongly expansive (even with a common strongly expansivity constant).

On the  other hand, recall the structure of $C^*$-algebra of $C(X)$ given by the standard sum and multiplication operations together with the norm
$$
\|\varphi\|_\infty=\sup_{x\in X}|\varphi(x)|.
$$
(Actually $\|\cdot\|_\infty$ may attain the value $\infty$.)
Denote by $Const(X)$ the set of constant observables.
Recall that $C\subset C(X)$ is a {\em subalgebra}
if
$$
\mathbb{C}\cdot C\cup(C+C)\cup(C\cdot C)\subset C.
$$
Though $\mathbb{C}\cdot SExp(f)\subset SExp(f)$, we don't have
$SExp(f)+SExp(f)\subset SExp(f)$.

However, there is another approach based on the following remark:
If a homeomorphism $f:X\to X$ is expansive, then for every observable $\varphi$ there is $\delta>0$ such that
$x,y\in X$ belong to the same level set of $\varphi$ if their orbits $(f^i(x))_{i\in\mathbb{Z}}$ and $(f^i(y))_{i\in\mathbb{Z}}$ stay $\delta$-close each other. In other words,
\begin{equation}
\label{papa}
x,y\in X\mbox{ and }d(f^i(x),f^i(y))\leq \delta\, (\forall i\in\mathbb{Z})
\quad\Longrightarrow\quad \varphi(x)=\varphi(y).
\end{equation}
The simple (but motivating) result below shows that this property characterizes the expansive homeomorphisms. More precisely, we have the following result.

\begin{theorem}
\label{ex1}
A homeomorphism of a metric space $f:X\to X$ is expansive if and only if there is $\delta>0$ such that \eqref{papa} holds $\forall \varphi\in C(X)$.
\end{theorem}

\begin{proof}
We only need to prove the sufficiency.
Suppose that there is $\delta>0$ satisfying \eqref{papa} for all $\varphi\in C(X)$.
If $f$ were not expansive,
there would exist sequences $x_n\neq y_n$ in $X$ such that
$d(f^i(x_n),f^i(y_n))\leq \frac{1}n$ for all $i\in\mathbb{N}$ and $n\in\mathbb{Z}$.
By taking $n\in\mathbb{N}$ with $\frac{1}n<\delta$ and
$\varphi:X\to \mathbb{C}$ defined by $\varphi(z)=d(x_n,z)$ for all $z\in X$ we get $\varphi\in C(X)$ which
does not satisfy \eqref{papa}. This is absurd so the result holds.
\end{proof}

This result motivates the following definition.

\begin{definition}
An observable $\varphi$ is {\em expansive} for $f$ if there is $\delta>0$
(called {\em expansivity constant}) such that \eqref{papa} holds.
\end{definition}

Then, we can rewrite the above theorem as follows:

\begin{customthm}{1(bis)}\label{ovni}
A homeomorphism of a metric space is expansive if and only if
every observable is expansive with a common expansivity constant.
\end{customthm}

There are two questions derived from this theorem.

\begin{question}
\label{q2}
What properties may have the set of expansive observables $Exp(f)$ of a given homeomorphism $f$?
\end{question}

\begin{question}
\label{q3}
What can happen if we drop the common expansivity constant hypothesis in Theorem \ref{ovni}?
\end{question}

The following remark holds.

\begin{remark}
Every strongly expansive observable $\varphi$ of a homeomorphism of a compact metric space $f:X\to X$ is expansive. In other words,
$$
SExp(f)\subset Exp(f).
$$
\end{remark}

\begin{proof}
Let $\epsilon$ be a strongly expansivity constant of $\varphi$.
Since $X$ is compact, $\varphi$ is uniformly continuous so there is $\delta>0$ such that
$$
|\varphi(a)-\varphi(b)|\leq\epsilon
\quad\mbox{ whenever }
a,b\in X,\, d(a,b)\leq\delta.
$$
Then, if $d(f^i(x),f^i(y))\leq\delta$ for all $i\in\mathbb{Z}$, $|\varphi(f^i(x))-\varphi(f^i(y))|\leq\epsilon$ so
$$
|\alpha_i(\varphi)(x)-\alpha_i(\varphi)(y)|\leq\epsilon,\quad\quad\forall i\in\mathbb{Z},$$
hence $\varphi(x)=\varphi(y)$
completing the proof.
\end{proof}

An observable $\varphi$ is {\em locally constant} on $X'\subset X$ if for every $x\in X'$ there is a neighborhood $U$ such that $\varphi(U\cap X')=\{\varphi(x)\}$. If $\varphi$ is locally constant on $X$ we just say that $\varphi$ is locally constant.

\begin{example}
Every expansive observable of the identity $id_X$ of $X$ is locally constant.
\end{example}

We will generalize this example later one.

\begin{example}
Every homeomorphism exhibiting an injective expansive observable is expansive.
\end{example}

We say that $f$ is {\em equicontinuous} if for every $\epsilon>0$ there is $\delta>0$ such that $d(f^i(x),f^i(y))\leq\epsilon$ for all $i\in\mathbb{Z}$ whenever $x,y\in X$ and $d(x,y)\leq \delta$.
We say that $f$ and another homeomorphism of a metric space $g:Y\to Y$ are {\em topologically conjugated}
if there is a homeomorphism $h:Y\to X$ such that $f\circ h=h\circ g$.

An algebra-valued map $E:Hom(X)\to 2^{C(X)}$ (where $Hom(X)$ is the set of homeomorphisms of $X$) is {\em invariant under topological conjugacy} if
the algebras $E(f)$ and $E(g)$ are isomorphic whenever $f,g\in Hom(X)$ are topologically conjugated.

We say that $x\in X$ is a periodic point of $f$ if $f^n(x)=x$ for some $n\in\mathbb{N}$.
Denote by $Per(f)$ the set of periodic points of $f$.
A subset of a topological space is F$_\sigma$ if it is the union of countably many closed subsets.

\begin{theorem}
\label{perra}
The following properties hold for all homeomorphisms of compact metric spaces $f:X\to X$ and $g:Y\to Y$:
\begin{enumerate}
\item
$Exp(f)$ is an F$_\sigma$ subalgebra of $C(X)$.
\item
$Const(X)\subset Exp(f)\subset C(X)$ and
either $Exp(f)=Const(X)$ or $Exp(f)=C(X)$ depending on whether $f$ is equicontinuous and $X$ is connected or else $f$ is expansive.
\item
The algebra-valued map $Exp:g\in Hom(X)\mapsto Exp(g)\subset C(X)$ is invariant under topological conjugacy.
\item
$Per(f)$ intersects at most countably many level sets of $\varphi$, $\forall \varphi\in Exp(f)$.
\item
If $X$ is infinite, then
for every $\varphi\in Exp(f)$ there are distinct $x,y\in X$ such that
\begin{equation}
\label{robador}
\lim_{n\to\infty}|\varphi(f^n(x))-\varphi(f^n(y))|=0.
\end{equation}
\end{enumerate}
\end{theorem}

This theorem provides a $C^*$-algebraic proof of the following well-known facts:

\begin{cor}
\label{c}
Every expansive homeomorphism of a compact metric space has countably many periodic points.
\end{cor}

A {\em continuum} is a nonempty compact metric space. A continuum is degenerated if it reduces to a singleton.

\begin{cor}
\label{c1}
The sole continuum exhibiting homeomorphisms which are expansive and equicontinuous simultaneously are the degenerated ones.
\end{cor}

About Item (5), it is motivated by Corollary 2.4.2 in \cite{bs}.
On the other hand, Item (2) implies that
the extreme cases
$Exp(f)=Const(X)$ and $Exp(f)=C(X)$ occur depending on whether $f$ is equicontinuous
and $X$ is connected or else $f$ is expansive.
A good question is when the reversed implications hold namely
if $Exp(f)=Const(X)$ or $Exp(f)=C(X)$ imply that $f$ is equicontinuous or expansive respectively.

Next, we deal with Question \ref{q3} namely
the question of which homeomorphism $f:X\to X$ satisfy $Exp(f)=C(X)$.
We may expect that such homeomorphisms and the expansive ones share some properties.
For instance, a homeomorphism $f:X\to X$ 
is {\em pointwise expansive} \cite{r} if for every $x\in X$ there is $\delta_x>0$ such that
$x=y$ whenever $x,y\in X$ satisfy $d(f^i(x),f^i(y))\leq \delta_x$ for all $i\in\mathbb{Z}$.
The difference between this definition and that of expansive actions is that the $\delta=\delta_x$ in the former depends on $x$ whereas it does not in the latter.
In particular, every expansive action is pointwise expansive but not conversely (see \cite{r} for a counterexample).

\begin{example}
Every homeomorphism of a metric space $f:X\to X$ satisfying $Exp(T)=C(X)$ is pointwise expansive.
\end{example}

\begin{proof}
Fix $x\in X$. Define $\varphi:X\to \mathbb{C}$ by
$\varphi(y)=d(x,y)$ for $y\in X$. Then, $\varphi\in C(X)=Exp(T)$ and so
$\varphi$ has an expansivity constant $\delta_x$.
If now $x,y\in X$ satisfy $d(f^i(x),f^i(y))\leq \delta$ for all $i\in\mathbb{Z}$,
then $\varphi(x)=\varphi(y)$ namely $0=d(x,y)$ thus $x=y$ hence $f$ is pointwise expansive.
\end{proof}

As a result,
there is no homeomorphism $f$ of the circle $S^1$ or the interval $[0,1]$
satisfying $Exp(f)=C(S^1)$ or $Exp(f)=C([0,1])$ respectively (see \cite{r}).
These facts motivate to study the homeomorphisms described below. 

\begin{definition}
A homeomorphism of a metric space $f:X\to X$ is
 {\em pseudoexpansive} (resp. {\em strongly pseudo expansive}) if $Exp(f)$ (resp. $SExp(f)$) is dense in $C(X)$.
 \end{definition}
 
The Stone-Weierstrass Theorem implies that $f$ is pseudoexpansive if and only if $Exp(f)$ separates points of $X$  (i.e. if
for all distinct $x,y\in X$ one has $\varphi(x)\neq\varphi(y)$ for some
expansive observable $\varphi$).
We have the following implications for homeomorphisms $f:X\to X$.\\

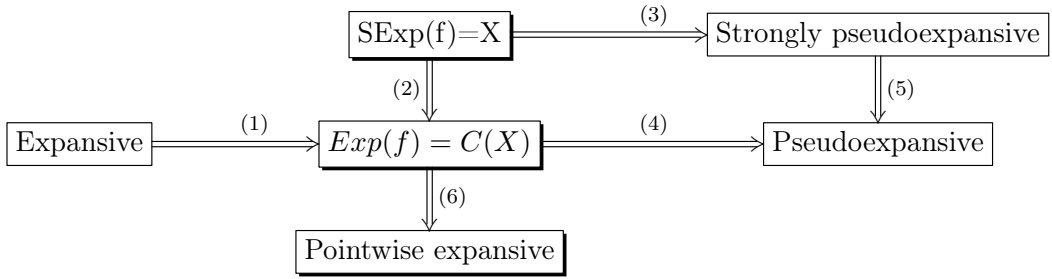
\centerline{
\xymatrix{ 
& & *+[F-,]{\txt{SExp(f)=X}} \ar@{=>}[d]_{(2)} \ar@{=>}[rr]^{(3)}
& & *+[F]{\txt{Strongly pseudoexpansive}} \\
*+[F]{\txt{Expansive}}  \ar@{=>}[rr]^{(1)}
& & *+[F-,]{\txt{$Exp(f)=C(X)$}}  \ar@{=>}[rr]^{(4)}
& & *+[F]{\txt{Pseudoexpansive}}\ar@{<=}[u]_{(5)} \\ 
& & *+[F-,]{\txt{Pointwise expansive}} \ar@{<=}[u]_{(6)}}\\
}

\medskip

\medskip

The following question was made by E. Rego:

\begin{question}
Is any of the converse of the implications (1) to (6) above true?
Is every pseudoexpansive homeomorphism pointwise expansive? 
\end{question}

With these definitions and motivated by the above results we will prove the following one.

\begin{theorem}
\label{furia}
There is no pseudoexpansive homeomorphisms of $S^1$ or $[0,1]$.
\end{theorem}

Since every expansive homeomorphism is pseudoexpansive,
we obtain a $C^*$-algebraic proof of the following well-known result:

\begin{cor}
There is no expansive homeomorphisms of $S^1$ or $[0,1]$.
\end{cor}

This paper is organized as follows. In Section \ref{sec2} we provide some preliminary results. In Section \ref{sec3}, we prove the theorems.

The authors would like to thank professor Rego for his useful observations.

\section{Preliminary lemmas}
\label{sec2}

\noindent
Hereafter $X$ is a compact metric space.
To prove Corollary \ref{c} we use the following lemma which seems to be well-known and whose proof is included for completeness.

\begin{lem}
\label{passagem}
If $P\subset X$ closed and
$\{t\in\mathbb{R} : P\cap \varphi^{-1}(t)\neq\emptyset\}$
is countable for every $\varphi\in C(X)$, then $P$ is countable.
\end{lem}

\begin{proof}
Suppose, by contradiction, that $P$ is uncountable. 
By the Cantor--Bendixson theorem \cite{k}, $P$ contains a nonempty perfect subset $C\subset P$. 
Every nonempty perfect compact metric space contains a subset homeomorphic to the standard Cantor set $K\subset [0,1]$; hence there exists a continuous embedding 
\[
g:K\hookrightarrow C\subset X.
\]
Let $Y=g(K)$, which is closed in $X$ and homeomorphic to $K$. 
Define a continuous map 
\[
\psi:Y\longrightarrow\mathbb{R}, \qquad 
\psi = i\circ g^{-1},
\]
where $i:K\hookrightarrow\mathbb{R}$ is the inclusion map. 
Then $\psi$ is continuous and injective on $Y$. 
Since $Y$ is closed in the normal space $X$, by the Tietze extension theorem there exists a continuous extension 
\[
\varphi:X\longrightarrow\mathbb{R}
\]
such that $\varphi|_Y = \psi$. 
Because $\varphi$ is injective on the uncountable set $Y\subset P$, the set of level values
\[
\{t\in\mathbb{R} : P\cap \varphi^{-1}(t)\neq\emptyset\}
\]
is uncountable. 
This contradicts the hypothesis. 
Hence $P$ is countable.
\end{proof}

Now, let $f:X\to X$ a
homeomorphism and $\varphi$ be an observable of $X$.
We say that $x\in X$ is a fixed point if $f(x)=x$.
Let $Fix(f)$ be the set of fixed points of $f$.

\begin{lem}
\label{sonriso}
If $\varphi$ is expansive for $f$, then there is $\delta>0$ such that
if $x,y\in Fix(T)$ and $d(x,y)<\delta$, then $\varphi(x)=\varphi(y)$.
\end{lem}

\begin{proof}
If $\delta$ is an expansivity constant of $\varphi$ and $d(x,y)\leq\delta$ for $x,y\in Fix(f)$, then
$d(f^n(x),f^n(y))<\delta$ for all $n\in\mathbb{Z}$ hence $\varphi(x)=\varphi(y)$.
\end{proof}

\begin{lem}
\label{culito}
If $\varphi$ is expansive for $f$, then
the number of level sets of $\varphi$ intersecting $Fix(f)$ is finite.
\end{lem}

\begin{proof}
Otherwise there are infinitely many level sets of $\varphi$ intersecting $Fix(f)$.
This results in a sequence of fixed points $x_1,x_2,\cdots$ of $f$ each one belonging to different level sets of $\varphi$ (i.e. $\varphi(x_i)\neq \varphi(x_j)$ for $i\neq j$).
Since $X$ compact, there are $j\neq i$ such
that $d(x_i,x_j)<\delta$ where $\delta$ is given by Lemma \ref{sonriso}.
Then, $d(f^n(x_i),f^n(x_j))=d(x_i,x_j)<\delta$ for every $n\in\mathbb{Z}$ so $\varphi(x_i)=\varphi(x_j)$ a contradiction.
\end{proof}

\begin{lem}
\label{mata}
$Exp(f^k)=Exp(f)$ for every $k\in\mathbb{Z}\setminus\{0\}$.
\end{lem}

\begin{proof}
Clearly $Exp(f)=Exp(f^{-1})$ by definition. So, we only need to prove the lemma
for $k\in\mathbb{N}$.
First we show $Exp(f)\subset Exp(f^k)$.
Take $\varphi\in Exp(f)$ and let $e$ be an expansivity constant of $\varphi$ w.r.t. $f$.  Let $\delta>0$ be such that
if $x,y\in X$ and $d(x,y)< \delta$, then $d(f^i(x),f^i(y))\leq e$ for all $0\leq i\leq k-1$.
Then, every $x,y\in X$ satisfying
$d(f^{ki}(x),f^{ki}(y))<\delta$ for all $i\in\mathbb{Z}$
satisfies $d(f^i(x),f^i(y))\leq e$ for all $i\in\mathbb{Z}$ hence
$\varphi(x)=\varphi(y)$ proving $\varphi\in E(f^k)$.
Henceforth $Exp(f)\subset Exp(f^k)$.

Now take $\varphi\in Exp(f^k)$
i.e. $\varphi$ is expansive w.r.t. $f^k$.
Let $e$ be the corresponding expansivity constant.
If $d(f^i(x),f^i(y))\leq e$ for all $i\in\mathbb{Z}$, then $d(f^{ki}(x),f^{ki}(y))\leq e$ for all $i\in\mathbb{Z}$ too so
$\varphi(x)=\varphi(y)$ thus $\varphi\in Exp(f)$ hence $Exp(f^k)\subset Exp(f)$ completing the proof.
\end{proof}

The next lemma characterizes the expansive observables for equicontinuous homeomorphisms.

\begin{lem}
\label{egon}
If $f$ is equicontinuous, then every $\varphi\in Exp(f)$ is locally constant.
\end{lem}

\begin{proof}
Take $\varphi\in Exp(f)$ and let $\epsilon$ be the corresponding expansivity constant.
For this $\epsilon$ let $\delta$ be given by the equicontinuity of $f$.
Take $x\in X$ and $U=B(x,\delta)$.
If $y\in U$, $d(f^n(x),f^n(y))\leq \epsilon$ for all $n\in\mathbb{Z}$ thus
$\varphi(x)=\varphi(y)$ proving $\varphi(U)=\{\varphi(x)\}$ completing the proof.
\end{proof}

From this lemma we obtain the following corollary.

\begin{cor}
\label{llevarlo}
If $f$ is equicontinuous and $X$ is connected, then $Exp(f)=Const(X)$.
\end{cor}

We shall prove that the set of expansive observables is a subalgebra of $C(X)$.

\begin{lem}
\label{lobo}
$Exp(f)$ is an F$_\sigma$ subalgebra of $C(X)$ containing $Const(X)$.
\end{lem}

\begin{proof}
That $Exp(f)$ contains $Const(X)$ follows directly from the definition.
Now, take $\lambda\in\mathbb{C}$ and $\varphi\in Exp(f)$.
If $\delta$ is an expansivity constant of $\varphi$ w.r.t. $f$
and $d(f^n(x),f^n(y))\leq\epsilon$ for all $n\in\mathbb{Z}$ then
$\varphi(x)=\varphi(y)$ so $\lambda\varphi(x)=\lambda\varphi(x)$
thus $\delta$ is also an expansivity constant for $\lambda\varphi$.
Therefore, $\lambda\varphi\in Exp(f)$ showing
$$
\mathbb{C}\cdot Exp(f)\subset Exp(f).
$$
Next, take $\varphi_1,\varphi_2\in Exp(f)$ with expansivity constants
$\delta_1,\delta_2$ respectively.
If $\delta=\min(\delta_1,\delta_2)$ and
$d(f^n(x),f^n(y))\leq\delta$ for all $n\in\mathbb{Z}$ then
$d(f^n(x),f^n(y))\leq\delta_1$ for all $n\in\mathbb{Z}$ so
$\varphi_1(x)=\varphi_1(y)$.
Likewise, $\varphi_2(x)=\varphi_2(y)$ hence
$$
(\varphi_1+\varphi_2)(x)=\varphi_1(x)+\varphi_2(x)=\varphi_1(y)+\varphi_2(y)=(\varphi_1+\varphi_2)(y)
$$
and
$$
(\varphi_1\cdot\varphi_2)(x)=\varphi_1(x)\varphi_2(x)=\varphi_1(y)\varphi_2(y)=(\varphi_1\cdot\varphi_2)(y).
$$
This shows
$$
Exp(f)+Exp(f)\subset Exp(f)\quad\mbox{ and }\quad Exp(f)\cdot Exp(f)\subset Exp(f).
$$
Therefore, $Exp(f)$ is a subalgebra.

Finally, we prove that $Exp(f)$ is F$_\sigma$. Define
$$
C_\delta(f)=\{\varphi\in C(X):\delta\mbox{ is an expansivity constant of $\varphi$ w.r.t.}f\},\quad\quad\forall \delta>0.
$$
Clearly
$$
Exp(f)=\bigcup_{n\in\mathbb{N}}C_{\frac{1}n}(f)
$$
so we only need to prove that $C_\delta(f)$ is closed for all $\delta>0$.
Take a sequence $\varphi_i\in C_\delta(f)$ converging to some observable
$\varphi$ in $C(X)$. Then,
$$
\lim_{i\to\infty}\varphi_i(z)=\varphi(z),\quad\quad\forall z\in X.
$$
Take $x,y\in X$ such that
$d(f^n(x),f^n(y))\leq\delta$ for all $n\in\mathbb{Z}$.
Then, $\varphi_i(x)=\varphi_i(y)$ for all $i\in\mathbb{N}$ and so
$$
\varphi(x)=\lim_{i\to\infty}\varphi_i(x)=\lim_{i\to\infty}\varphi_i(y)=\varphi(y)
$$
thus $\delta$ is an expansivity constant of $\varphi$.
This proves $\varphi\in C_\delta(f)$ therefore $C_\delta(f)$ is closed as asserted.
\end{proof}

Next, we deal with topologically conjugated homeomorphisms.

\begin{lem}
\label{wrong}
Let $f:X\to X$ and $g:Y\to Y$ be topologically conjugated homeomorphisms of compact metric spaces. Then, the subalgebras
$Exp(f)$ and $Exp(g)$ are isomorphic by an isomorphism carrying constant functions into constant functions.
\end{lem}

\begin{proof}
Let $h:Y\to X$ be a homeomorphism such that
$f\circ h=h\circ g$.
Define $H:C(X)\to C(Y)$ by
$H(\varphi)=\varphi\circ h$.
Clearly $H$ is an isomorphism with inverse $H^{-1}:C(Y)\to C(X)$ given by $H^{-1}(\psi)=\psi\circ h^{-1}$.
Clearly $H(Const(X))=Const(Y)$.
Then, we only have to prove that $H(Exp(f))\subset Exp(g)$.

Take $\varphi\in Exp(f)$ and let $\epsilon$ be an expansivity constant.
Let $\delta$ be such that $d(a,b)\leq\delta$ implies
$d(h(a),h(b))\leq\epsilon$.
Suppose that $z,w\in Y$ and $d(g^i(z),g^i(w))\leq\delta$ for all $i\in\mathbb{Z}$.
Then, $d(h(g^i(z)),h(g^i(w)))\leq\epsilon$ and so
$d(f^i(h(z)),f^i(h(w)))\leq\epsilon$ for all $i\in\mathbb{Z}$
henceforth $\varphi(h(z))=\varphi(h(w))$ thus
$H(\varphi)(z)=H(\varphi)(w)$.
This proves that $\delta$ is an expansivity constant of $H(\varphi)$ w.r.t. $g$
hence $H(\varphi)\in Exp(g)$. Therefore, $H(Exp(f))\subset Exp(g)$ completing the proof.
\end{proof}

The {\em nonwandering set} of $f$ is defined by
$$
\Omega(f)=\{x\in X:\forall\mbox{ neighborhood }U \mbox{ of }x\, \exists n\in\mathbb{N}\mbox{ s.t. }f^n(U)\cap U\neq\emptyset\}.
$$

\begin{lem}
\label{ulrich}
If $X$ is locally connected and $f$ satisfies the following property:
\begin{enumerate}
\item[(P)]
$diam(f^n(O))\to0$ as $n\to\infty$ whenever $O\subset X$ is an open connected set such that the collection $\{f(O),f^2(O),\cdots f^n(O)\cdots\}$  is pairwise disjoint,
\end{enumerate}
then every expansive observable of $f$ is locally constant on $X\setminus \Omega(f)$.
\end{lem}

\begin{proof}
Let $\varphi$ be an expansive observable with expansivity constant $\delta$.
If $x\in X\setminus\Omega(f)$ there is a connected open neighborhood $O$ of $x$ such that the sequence $(f^n(O))_{n\geq\mathbb{N}}$ is pairwise disjoint.
Then,
$diam(f^n(O))\to 0$ as $n\to\infty$ by (P) and so there is $N\in\mathbb{N}$ such that
$diam(f^n(O))<\delta$ whenever $n\in\mathbb{Z}$ with $|n|\geq N$.
Since $f$ is continuous, there exists also a neighborhood $U$ of $x$ in $O$ such that
$diam(f^n(U))\leq \delta$ for every $-N\leq n\leq N$.
Then, $diam(f^n(U))\leq\delta$ for all $n\in\mathbb{Z}$ and so
$\varphi$ is constant in $U$ since $\delta$ is an expansivity constant of
$\varphi$. This completes the proof.
\end{proof}

The {\em stable} and {\em unstable sets} of $f$ at $x\in X$ are defined by
$$
W^s(x)=\{y\in X:\lim_{n\to \infty}d(f^n(x),f^n(y))=0\}
$$
and
$$
W^u(x)=\{y\in X:\lim_{n\to- \infty}d(f^n(x),f^n(y))=0\}.
$$
Analogously, for every observable $\varphi$ we define
$$
W^s(x,\varphi)=\{y\in X:\lim_{n\to \infty}|\varphi(f^n(x))-\varphi(f^n(y))|=0\}
$$
and
$$
W^u(x,\varphi)=\{y\in X:\lim_{n\to -\infty}|\varphi(f^n(x))-\varphi(f^n(y))|=0\}.
$$
It follows that
$$
W^*(x)=\bigcap_{\varphi\in C(X)}W^*(x,\varphi),\quad\quad\forall x\in X\mbox{ and }*=s,u.
$$
If $\epsilon>0$ we define the dynamical balls
$$
W^s_\epsilon(x)=\{y\in X:d(f^n(x),f^n(y))\leq\epsilon,\, \forall n\geq0\}
$$
and
$$
W^u_\epsilon(x)=\{y\in X:d(f^n(x),f^n(y))\leq\epsilon,\, \forall n\leq0\}.
$$

\begin{lem}
\label{inteligencia}
For every $\varphi\in Exp(f)$ there is $\epsilon>0$ such that
$W^*_\epsilon(x)\subset W^*(x,\varphi)$ for all $x\in X$ and $*=s,u$.
\end{lem}

\begin{proof}
Let $\epsilon$ be the expansivity constant of $\varphi$.
If the conclusion of the lemma fails for $*=s$ and this $\epsilon$, then
there is $y\in W^s_\epsilon(x)\setminus W^s(x,\varphi)$.
It follows that
$d(f^n(x),f^n(y))\leq\epsilon$ for all $n\geq0$ but there are $\delta>0$
and a sequence $n_i\to\infty$ such that
$\delta\leq |\varphi(f^{n_i}(x))-\varphi(f^{n_i}(y))|$ for all $i\in\mathbb{N}$.
Since $X$ is compact, we can assume that
$f^{n_i}(x)\to \hat{x}$ and $f^{n_i}(y)\to \hat{y}$ as $i\to\infty$ for some
$\hat{x},\hat{y}\in X$.
Since $\varphi$ is continuous,
$|\varphi(\hat{x})-\varphi(\hat{y})|\geq\delta$ hence
$\varphi(\hat{x})\neq \varphi(\hat{y})$.
However,
for all $m\in\mathbb{Z}$ one has that
$m+n_i>0$ for $i$ large hence
$$
d(f^m(\hat{x}),f^m(\hat{y}))=\lim_{i\to\infty}d(f^{m+n_i}(x),f^{m+n_i}(y))\leq\epsilon
$$
hence $\varphi(\hat{x})=\varphi(\hat{y})$ a contradiction.
This completes the proof for $*=s$. The proof is analogous for $*=u$.
\end{proof}

\section{Proof of corollaries \ref{c}, \ref{c1} and theorems \ref{perra} and \ref{furia}}
\label{sec3}

\begin{proof}[Proof of Theorem \ref{perra}]
Item (1) to (3) follow from Lemma \ref{lobo}, Corollary \ref{llevarlo}
and Lemma \ref{wrong} respectively.
To prove Item (4) note that
\begin{equation}
\label{es}
Per(f)=\bigcup_{k\in\mathbb{N}}Fix(f^k).
\end{equation}
If $\varphi$ is an expansive observable of $f$, then it is also expansive for $f^k$ for all $k\in\mathbb{N}$ by Lemma \ref{mata}.
Then, $Fix(f^k)$ intersects finitely many level sets of $\varphi$ only by Lemma \ref{culito}. Therefore, $Per(f)$ intersects at most countable many
level sets of $\varphi$ by \eqref{es}. This completes the proof of Item (4).
To prove Item (5), we take $\varphi\in Exp(f)$ and $\epsilon$ as in Lemma \ref{inteligencia}.
By Proposition 2.4.1 in \cite{bs} there are distinct $x,y\in X$ such that
$y\in W^s_\epsilon(x)$. Then, $y\in W^s(x,\varphi)$ by Lemma \ref{inteligencia}
and so \eqref{robador} holds. This completes the proof.
\end{proof}

\begin{proof}[Proof of Corollary \ref{c}]
Let $f:X\to X$ be an expansive homeomorphism of a compact metric space.
If $Per(f)$ is uncountable, then $P=Fix(f^k)$ which is closed (by the continuity of $f$) is also uncountable for some $k\in\mathbb{N}$ by \eqref{es}.
On the other hand, $Exp(f)=C(X)$ by Item (2) of Theorem \ref{perra} so $P$ intersects at most countable many level sets of $\varphi$ for all $\varphi\in C(X)$ (by Item (4) of Theorem \ref{perra}). Since $P$ is closed, it is countable by Lemma \ref{passagem} a contradiction. Therefore, $Per(f)$ is countable.
\end{proof}

\begin{proof}[Proof of Corollary \ref{c1}]
Let $X$ be a continuum exhibiting a homeomorphism $f:X\to X$ which is both expansive and equicontinuous. Then, $C(X)=Exp(f)=Const(X)$ by Theorem \ref{perra} and so $X$ reduces to a single point.
\end{proof}

\begin{proof}[Proof of Theorem \ref{furia}]
First note that every homeomorphism $f:S^1\to S^1$ satisfies (P) in Lemma  \ref{ulrich}.

If $\Omega(f)=S^1$, $f$ is topologically conjugated to a rotation $R$.
Since $R$ is an isometry, it is equicotinuous so $C(R)=Const(S^1)$
by Item (2) of Theorem \ref{perra} thus $Exp(f)=Const(S^1)$ by Item (2) of Theorem \ref{perra}.
Therefore, $f$ cannot be pseudoexpansive.

If $\Omega(f)\neq S^1$, then $S^1\setminus \Omega(f)$ consists of an at most countably family of disjoint open intervals $J$.
Since $f$ satisfies (P), every $\varphi\in Exp(f)$ is constant in $J$ by Lemma \ref{ulrich} so $Exp(f)$ cannot be dense in $C(S^1)$.
Therefore, $f$ cannot be pseudoexpansive in this case too.
This completes the proof.
\end{proof}

\section*{Declaration of competing interest}

\noindent
There is no competing interest.

\section*{Data availability}

\noindent
No data was used for the research described in the article.

\section*{Acknowledgements}

\noindent
WJ was partially supported by the National Research Foundation (NRF) of the Republic of Korea (MSIT) (No. NRF-2021R1F1A1052631).

\end{document}